\documentclass[preprint]{imsart}
\pdfoutput=1
\RequirePackage[OT1]{fontenc}
\usepackage{amsthm, amsmath, natbib, amssymb, bm, enumitem}
\RequirePackage[colorlinks,citecolor=blue,urlcolor=blue]{hyperref}
\usepackage{booktabs}
\usepackage{mathtools} 
\mathtoolsset{showonlyrefs}
\startlocaldefs
\theoremstyle{plain}
\newtheorem{thm}{Theorem}%[section]
\newtheorem{lemma}{Lemma}%[section]
\endlocaldefs

\newcommand{\medx}{\mathrm{med}\,x}
\newcommand{\medz}{\mathrm{med}\,z}

\DeclareMathOperator*{\argmin}{arg\,min}
\allowdisplaybreaks
\begin{document}
\begin{frontmatter}
\title{A sharper bound of the Hotelling-Solomons inequality}
\runtitle{A sharper bound of Hotelling-Solomons inequality}

\begin{aug}
\author{\fnms{Yuzo} \snm{Maruyama}
\ead[label=e1]{maruyama@port.kobe-u.ac.jp}%
}

\runauthor{Y. Maruyama}

\address{Kobe University \\
\printead{e1}}

\end{aug}

\begin{abstract}
The original Hotelling-Solomons inequality states that an upper bound of the absolute difference between the mean and median, standardized by the standard deviation, is $1$.
However, in this paper, we introduce a new bound that depends on the sample size, which is strictly smaller than $1$.
\end{abstract}


\end{frontmatter}

\section{Introduction}
\label{sec:intro}
Suppose we have a set of $n$ numbers,
\begin{equation}\label{sample}
 x_1,\dots,x_n\quad\text{with}\quad x_1\leq x_2 \leq \dots \leq x_n \text{ and } x_1<x_n.
\end{equation}
The mean, the standard deviance and the median are
\begin{equation*}
\begin{split}
 \bar{x}&=\frac{\sum_{i=1}^n x_i}{n}, \ s=\sqrt{\displaystyle\frac{\sum_{i=1}^n(x_i-\bar{x})^2}{n}},\\
\medx &=
\begin{cases}
 x_{k+1} & n=2k+1 \\
(x_{k}+x_{k+1})/2 & n=2k
\end{cases}
\quad\text{for }k\in\mathbb{N}.
\end{split}
\end{equation*}
As in the above, the sample size $n$ can be represented as 
$2k+1$ when it is odd, and as $2k$ when it is even, where $k$ is a positive integer.
Using the well-known fact 
\begin{equation}\label{argmin}
 \medx=\argmin_{\mu \in\mathbb{R}}\sum_{i=1}^n|x_i-\mu|,
\end{equation}
we easily have the following inequality
\begin{equation}\label{ineq.1}
 \begin{split}
 |\medx-\bar{x}|&=\frac{1}{n}\bigl|\sum_{i=1}^n(x_i-\medx)\bigr|
\leq \frac{1}{n}\sum_{i=1}^n\bigl|x_i-\medx\bigr|\\
&\leq \frac{1}{n}\sum_{i=1}^n\bigl|x_i-\bar{x}\bigr|
\leq \sqrt{\frac{1}{n}\sum\nolimits_{i=1}^n(x_i-\bar{x})^2}
%\sqrt{\displaystyle\frac{1}{n}\sum_{i=1}^n(x_i-\bar{x})^2} \\
=s,
\end{split} 
\end{equation}
where the second and third inequalities are from \eqref{argmin} and Jensen's inequality, respectively.
The inequality equivalent to \eqref{ineq.1},
\begin{equation}\label{skewness}
 \frac{|\medx-\bar{x}|}{s}\leq 1
\end{equation}
is called \cite{Hotelling-Solomons-1932} inequality; 
they regarded $(\medx-\bar{x})/s$ as a measure of skewness.
\cite{Garver-1932} provides a proof using some algebraic inequalities.

\cite{Majindar-1962} investigated the corresponding inequality for random variable $X$, with a positive standard deviation $\sqrt{\mathrm{Var}[X]}$.
For the median $M$ and the expected value $\mathrm{E}[X]$,
\cite{Majindar-1962} proved the following inequality
\begin{equation}
 \frac{\left|\mathrm{E}[X]-M\right|}{\sqrt{\mathrm{Var}[X]}}<2\Bigl(\frac{pq}{p+q}\Bigr)^{1/2} ,
\end{equation}
where $p=\mathrm{Pr}(X>\mathrm{E}[X])$ and $q=\mathrm{Pr}(X<\mathrm{E}[X])$.
As in Lemma \ref{lem:pq} below, we have
\begin{equation}
 2\Bigl(\frac{pq}{p+q}\Bigr)^{1/2}\leq \min\bigl(2\{p(1-p)\}^{1/2},2\{q(1-q)\}^{1/2},(p+q)^{1/2}\bigr)
\leq 1,
\end{equation}
which is a sharper bound if either $p\in(0,1/2)$ or $q\in(0,1/2)$ are provided. 
Further, for the set of $n$ numbers given by \eqref{sample}, 
\cite{Majindar-1962} suggested a without proof sharper bound than \eqref{skewness},
given by
\begin{equation}\label{odd.1}
 \frac{|\medx-\bar{x}|}{s}\leq\Bigl(\frac{k}{k+1} \Bigr)^{1/2},
\end{equation}
when $n$ is odd as $ n=2k+1$ with $k\in\mathbb{N}$.
\cite{Hawkins-1971} and \cite{Boyd-1971} investigated some properties
of the ordered statistic and showed that the range of the ordered statistic, \eqref{sample},
should be
\begin{equation}
\begin{split}
\displaystyle -(n-1)^{1/2} &\leq  \frac{x_1-\bar{x}}{s}\leq -\Bigl(\frac{1}{n-1} \Bigr)^{1/2},\\
\displaystyle -\Bigl(\frac{n-i}{i} \Bigr)^{1/2} &\leq  
\frac{x_i-\bar{x}}{s}\leq \Bigl(\frac{i-1}{n+1-i} \Bigr)^{1/2}, \\
\displaystyle \Bigl(\frac{1}{n-1} \Bigr)^{1/2}&\leq  \frac{x_n-\bar{x}}{s}\leq (n-1)^{1/2}. 
\end{split}
\end{equation}
Hence, when $n$ is odd as $ n=2k+1$ with $k\in\mathbb{N}$,
the median, $x_{k+1}$, satisfies \eqref{odd.1}.
Further they showed that the equalities
\begin{equation}
  \frac{\medx-\bar{x}}{s}=
\begin{cases}
 -\bigl\{k/(k+1) \bigr\}^{1/2} \\
 \bigl\{k/(k+1) \bigr\}^{1/2} 
\end{cases}
\end{equation}
are attained by the cases
\begin{align}
x_1&= \dots = x_{k+1}<x_{k+2}=\dots =x_{2k+1}, \label{eq.odd.1}\\
x_1&= \dots = x_k<x_{k+1}=\dots =x_{2k+1},\label{eq.odd.2}
\end{align}
respectively. 
Their methodology, however, does not seem applicable for $n=2k$.

In this article, following \cite{Majindar-1962},
we give the proof of \eqref{odd.1} for $n=2k+1$ and find a new bound for $n=2k$.
\begin{thm}\label{thm:main}
Suppose $n\geq 3$. Then
\begin{equation}\label{eq:main}
 \frac{|\medx-\bar{x}|}{s}\leq
\displaystyle
\begin{cases}
\displaystyle \Bigl(\frac{k}{k+1} \Bigr)^{1/2} & n=2k+1, \\
\displaystyle \Bigl(\frac{k-1}{k+1} \Bigr)^{1/2} & n=2k.
\end{cases}
\end{equation}
Further the equalities for $n=2k$, 
\begin{equation}
  \frac{\medx-\bar{x}}{s}=-\Bigl(\frac{k-1}{k+1} \Bigr)^{1/2},\quad
  \frac{\medx-\bar{x}}{s}=\Bigl(\frac{k-1}{k+1} \Bigr)^{1/2}
\end{equation}
are attained by the cases
\begin{align}
x_1&= \dots =x_k= x_{k+1}<x_{k+2}=\dots =x_{2k}, \label{eq.even.1} \\
x_1&= \dots = x_{k-1}<x_{k}=x_{k+1}=\dots =x_{2k},\label{eq.even.2}
\end{align}
respectively. 
\end{thm}
The proof of Theorem \ref{thm:main} is given in Section \ref{sec:proof}.
\section{Proof of Theorem 1}
\label{sec:proof}
Let $z_i=(x_i-\bar{x})/s$. 
Then the left-hand side of \eqref{eq:main}
is replaced by $ |\medz|$ where
\begin{equation}
\medz =
\begin{cases}
 z_{k+1} & n=2k+1, \\
(z_{k}+z_{k+1})/2 & n=2k.
\end{cases} 
\end{equation}
Since $\sum_{i=1}^n z_i=0$ and $z_1<z_n$, 
there exist $a>0$, $\ell \in\mathbb{N}$ and $m\in\mathbb{N}$ with $ 1\leq \ell\leq m \leq n-1$ 
such that
\begin{equation}
z_i<0 \text{ for }1\leq i\leq \ell, \quad z_i>0 \text{ for }m+1\leq i\leq n, \quad
\text{and}\quad a= -\sum_{i=1}^\ell z_i = \sum_{i=m+1}^n z_i.
\end{equation}
If $l<m$, $z_{\ell+1}=\dots=z_m=0$.
Then, by the Cauchy-Schwarz inequality, we have
\begin{equation}
\begin{split}
a^2&=\Bigl(\sum_{i=1}^\ell z_i\Bigr)^2\leq \ell \sum_{i=1}^\ell z_i^2, \\
a^2&=\Bigl(\sum_{i=m+1}^n z_i\Bigr)^2\leq (n-m) \sum_{i=m+1}^n z_i^2, 
\end{split} 
\end{equation}
and hence
\begin{equation}\label{inequ:1}
\begin{split}
& \left(\frac{n}{\ell}+\frac{n}{n-m}\right) a^2\leq n\Bigl\{\sum_{i=1}^\ell+ \sum_{i=m+1}^n\Bigr\}z_i^2  
=n\sum_{i=1}^n z_i^2=n^2 \\
&\text{or equivalently}\quad a^2\leq  \frac{n^2}{n/\ell+n/(n-m)}.
\end{split}
\end{equation}

Suppose $n$ is odd as $n=2k+1$. If $z_{k+1}=0$, the inequality \eqref{eq:main} follows.
Assume $z_{k+1}>0$. Then $m\leq k$.
By \eqref{inequ:1} and Lemma \ref{lem:pq} below,
\begin{equation}\label{a2}
 a^2\leq 
n^2\frac{m}{n}\left(1-\frac{m}{n}\right) 
\leq n^2\frac{k}{n}\left(1-\frac{k}{n}\right)=k(k+1), 
\end{equation}
where the second inequality follows from the fact $m\leq k$.
Further we have
\begin{equation}\label{kisu.00}
 a=\sum_{i=m+1}^n z_i\geq \sum_{i=k+1}^n z_i \geq (k+1)z_{k+1}
\end{equation}
and hence
\begin{equation}\label{kisu.1} 
 z_{k+1} \leq \frac{a}{k+1} \leq \Bigl(\frac{k}{k+1} \Bigr)^{1/2}. 
\end{equation}
The equalities in \eqref{a2}, \eqref{kisu.00} and \eqref{kisu.1},
\begin{equation}
 a=\{k(k+1)\}^{1/2}, \ a=(k+1)z_{k+1}, \ \text{and} \ z_{k+1}= \Bigl(\frac{k}{k+1} \Bigr)^{1/2}
\end{equation} 
are attained by
\begin{equation}
 z_1=\dots=z_k =-\Bigl(\frac{k+1}{k} \Bigr)^{1/2} \ \text{ and } \ 
z_{k+1}=\dots=z_{2k+1} =\Bigl(\frac{k}{k+1}\Bigr)^{1/2},
\end{equation}
which corresponds to the case \eqref{eq.odd.2}.
Assume $z_{k+1}<0$. Then $k+1\leq \ell$. 
By \eqref{inequ:1} and Lemma \ref{lem:pq} below,
\begin{equation}\label{a2.0}
 a^2
\leq n^2\frac{\ell}{n}\left(1-\frac{\ell}{n}\right) 
\leq n^2\frac{k+1}{n}\left(1-\frac{k+1}{n}\right)=k(k+1), 
\end{equation}
where the second inequality follows from the fact $k+1\leq\ell$.
Further we have
\begin{equation}\label{kisu.000}
 a=-\sum_{i=1}^\ell z_i\geq -\sum_{i=1}^{k+1} z_i \geq -(k+1)z_{k+1}
\end{equation}
and hence
\begin{equation}\label{kisu.2} 
0< -z_{k+1} \leq \frac{a}{k+1} \leq \Bigl(\frac{k}{k+1} \Bigr)^{1/2}. 
\end{equation}
The equalities in \eqref{a2.0}, \eqref{kisu.000} and \eqref{kisu.2},
\begin{equation}
 a=\{k(k+1)\}^{1/2}, \ a=-(k+1)z_{k+1}, \ \text{and} \ z_{k+1}= -\Bigl(\frac{k}{k+1} \Bigr)^{1/2}
\end{equation} 
are attained by
\begin{equation}
 z_1=\dots=z_{k+1} =-\Bigl(\frac{k}{k+1} \Bigr)^{1/2} \ \text{ and } \ 
z_{k+2}=\dots=z_{2k+1} =\Bigl(\frac{k+1}{k}\Bigr)^{1/2},
\end{equation}
which corresponds to the case \eqref{eq.odd.1}.
By \eqref{kisu.1} and \eqref{kisu.2}, we complete the proof for the case $n=2k+1$.

Suppose $n$ is even as $n=2k$. If $z_{k}+z_{k+1}=0$, the inequality \eqref{eq:main} follows.
Then we consider the four cases
\begin{equation}
 \begin{cases}
  \text{Case I} & z_{k+1} > - z_{k} \geq 0, \\
  \text{Case II} & -z_k > z_{k+1}\geq 0,  \\
  \text{Case III} & 0< z_{k} \leq z_{k+1}, \\
  \text{Case IV} & z_{k}\leq z_{k+1}<0.
 \end{cases}
\end{equation}

For Case I, we have $m= k$ and, by Lemma \ref{lem:pq} below,
\begin{equation}
 a^2 \leq n^2\frac{m}{n}\left(1-\frac{m}{n}\right)=
n^2\frac{k}{n}\left(1-\frac{k}{n}\right)=\frac{n^2}{4}=k^2.
\end{equation}
Further we have
\begin{equation}
 a=\sum_{i=m+1}^n z_i= \sum_{i=k+1}^n z_i \geq k z_{k+1}
\end{equation}
and hence
\begin{equation}\label{gusu.1}
0<\frac{z_k+z_{k+1}}{2}\leq \frac{z_{k+1}}{2} \leq \frac{a}{2k} \leq \frac{1}{2}.
\end{equation}

For Case II, we have $\ell= k$ and
\begin{equation}
 a^2 \leq n^2\frac{\ell}{n}\left(1-\frac{\ell}{n}\right)=
n^2\frac{k}{n}\left(1-\frac{k}{n}\right)=\frac{n^2}{4}=k^2.
\end{equation}
Further we have
\begin{equation}
 a=-\sum_{i=1}^\ell z_i= -\sum_{i=1}^k z_i \geq -k z_{k}
\end{equation}
and hence
\begin{equation}\label{gusu.2}
0<-\frac{z_k+z_{k+1}}{2}\leq -\frac{z_k}{2} \leq \frac{a}{2k} \leq \frac{1}{2}.
\end{equation}

For Case III, we have $m+1 \leq k $ and
\begin{equation}\label{a2.gusu}
 a^2\leq n^2\frac{m}{n}\left(1-\frac{m}{n}\right)\leq n^2\frac{k-1}{n}\left(1-\frac{k-1}{n}\right)
=(k+1)(k-1).
\end{equation}
Further
\begin{equation}\label{gusu.00}
 a=\sum_{i=m+1}^n z_i\geq\sum_{i=k}^n z_i =\sum_{i=k+1}^n z_i +z_k\geq kz_{k+1} +z_k.
\end{equation}
Hence we have
\begin{equation}\label{gusu.3}
0< \frac{z_k+z_{k+1}}{2}\leq \frac{z_k+ k z_{k+1}}{k+1}\leq \frac{a}{k+1}\leq 
\Bigl(\frac{k-1}{k+1} \Bigr)^{1/2}. 
\end{equation}
The equalities in \eqref{a2.gusu}, \eqref{gusu.00} and \eqref{gusu.3},
\begin{equation}
 a=\{(k-1)(k+1)\}^{1/2}, \ a=kz_{k+1} +z_k, \ \text{and} \ 
\frac{z_k+z_{k+1}}{2}=\Bigl(\frac{k-1}{k+1}\Bigr)^{1/2}
\end{equation} 
are attained by
\begin{equation}
z_1=\dots=z_{k-1} =-\Bigl(\frac{k+1}{k-1} \Bigr)^{1/2} \ \text{ and } \ 
z_{k}=\dots=z_{2k} =\Bigl(\frac{k-1}{k+1}\Bigr)^{1/2},
\end{equation}
which corresponds to the case \eqref{eq.even.2}.

For Case IV, we have $\ell \geq k+1 $ and
\begin{equation}\label{a2.gusu.00}
 a^2\leq n^2\frac{\ell}{n}\left(1-\frac{\ell}{n}\right)\leq n^2\frac{k+1}{n}\left(1-\frac{k+1}{n}\right)=(k+1)(k-1).
\end{equation}
Further
\begin{equation}\label{gusu.000}
 a=- \sum_{i=1}^\ell z_i \geq - \sum_{i=1}^{k+1} z_i = -\sum_{i=1}^k z_i - z_{k+1} \geq -kz_k -z_{k+1}.
\end{equation}
Hence we have
\begin{equation}\label{gusu.4}
 -\frac{z_k+z_{k+1}}{2}\leq -\frac{ k z_k + z_{k+1}}{k+1}\leq \frac{a}{k+1}\leq 
\Bigl(\frac{k-1}{k+1} \Bigr)^{1/2}. 
\end{equation}
The equalities in \eqref{a2.gusu.00}, \eqref{gusu.000} and \eqref{gusu.4},
\begin{equation}
 a=\{(k-1)(k+1)\}^{1/2}, \ a=-(kz_k +z_{k+1}), \ \text{and} \ 
\frac{z_k+z_{k+1}}{2}=-\Bigl(\frac{k-1}{k+1} \Bigr)^{1/2}
\end{equation} 
are attained by
\begin{equation}
 z_1=\dots=z_{k+1} =-\Bigl(\frac{k-1}{k+1} \Bigr)^{1/2} \ \text{ and } \ 
z_{k+2}=\dots=z_{2k} =\Bigl(\frac{k+1}{k-1}\Bigr)^{1/2},
\end{equation}
which corresponds to the case \eqref{eq.even.1}.

Note $\{(k-1)/(k+1)\}^{1/2}>1/2$ for $k\geq 2$. Then, by \eqref{gusu.1}, \eqref{gusu.2}, 
\eqref{gusu.3}, and \eqref{gusu.4}, we have
\begin{equation}
 \Bigl|\frac{z_k+z_{k+1}}{2}\Bigr|
\leq  \Bigl(\frac{k-1}{k+1} \Bigr)^{1/2},
\end{equation}
which completes the proof for the case $n=2k$.

\medskip

\begin{lemma}\label{lem:pq}
 Suppose $ p>0$, $q>0$ and $p+q\leq 1$.
Then
\begin{equation}
 \frac{1}{1/p+1/q}\leq \min\Bigl\{p(1-p),q(1-q), \frac{p+q}{4}\Bigr\}.
\end{equation}
\end{lemma}
\begin{proof}
 We have
\begin{equation}
 \begin{split}
  p(1-p)-\frac{pq}{p+q}&=\frac{p^2}{p+q}(1-p-q)\geq 0, \\
  q(1-q)-\frac{pq}{p+q}&=\frac{q^2}{p+q}(1-p-q)\geq 0, \\
\frac{p+q}{4}-\frac{pq}{p+q}&=\frac{(p-q)^2}{4(p+q)}\geq 0,
 \end{split}
\end{equation}
which completes the proof.
\end{proof}

\end{document}